\documentclass[11pt,a4paper,english]{smfart}

 \usepackage[margin=1.15in]{geometry}
\usepackage{amscd,amssymb, amsmath,amsfonts, wasysym, mathrsfs, enumitem, stmaryrd, mathtools,hhline,color,enumitem, tikz-cd}
\usepackage{graphicx}
\usepackage[all, cmtip]{xy}
\usepackage[T1]{fontenc}
\usepackage{newtxtext}
\usepackage{newtxmath}
\usepackage{textcomp}
\usepackage{url}
\usepackage{array}

\usepackage{thmtools}
\usepackage{enumitem}
\usepackage[pagebackref=true]{hyperref} 
\usepackage{nameref}
\usepackage{cleveref}
\usepackage{comment}
\usetikzlibrary{arrows,positioning,calc,intersections}
\setlist[itemize]{wide = 0pt, labelwidth = 2em, labelsep*=0em, itemindent = 0pt, leftmargin = \dimexpr\labelwidth + \labelsep\relax, noitemsep,topsep = 1ex,}
\setlist[enumerate]{wide = 0pt, labelwidth = 2em, labelsep*=0em, itemindent = 0pt, leftmargin = \dimexpr\labelwidth + \labelsep\relax, noitemsep,topsep = 1ex}

\definecolor{hot}{RGB}{65,105,225}

\theoremstyle{plain}
\newtheorem{theorem}{Theorem}[section]
\newtheorem{proposition}[theorem]{Proposition}
 
\newtheorem{thmx}{Theorem}
\renewcommand{\thethmx}{\Alph{thmx}}

\newtheorem{claim}[theorem]{Claim}

\newtheorem{corollary}[theorem]{Corollary}
\newtheorem{conj}[theorem]{Conjecture}
\newtheorem{lemma}[theorem]{Lemma}
\theoremstyle{definition}

\newtheorem{definition}[theorem]{Definition}

\newtheorem{remark}[theorem]{Remark}

\newtheorem*{ex*}{Example}



\usepackage[hyperpageref]{backref}

\theoremstyle{plain}
\newlist{thmlist}{enumerate}{1}
\setlist[thmlist]{wide = 0pt, labelwidth = 2em, labelsep*=0em, itemindent = 0pt, leftmargin = \dimexpr\labelwidth + \labelsep\relax, noitemsep,topsep = 1ex, font=\normalfont, label=(\roman*), ref=\thetheorem.(\roman{thmlisti})}

\newlist{thmenum}{enumerate}{1} 
\setlist[thmenum]{wide = 0pt, labelwidth = 2em, labelsep*=0em, itemindent = 0pt, leftmargin = \dimexpr\labelwidth + \labelsep\relax, noitemsep,topsep = 1ex, font=\normalfont, label=(\roman*), ref=\thethmx.(\roman{thmenumi})}
\crefalias{thmenumi}{thmx} 

\newlist{corlist}{enumerate}{1} 
\setlist[corlist]{wide = 0pt, labelwidth = 2em, labelsep*=0em, itemindent = 0pt, leftmargin = \dimexpr\labelwidth + \labelsep\relax, noitemsep,topsep = 1ex, font=\normalfont, label=(\roman*), ref=\thecorx.(\roman{corlisti})}
\crefalias{corlisti}{corx} 

\addtotheorempostheadhook[theorem]{\crefalias{thmlisti}{theorem}}

\addtotheorempostheadhook[corollary]{\crefalias{thmlisti}{corollary}}

\addtotheorempostheadhook[proposition]{\crefalias{thmlisti}{proposition}}

\addtotheorempostheadhook[definition]{\crefalias{thmlisti}{definition}}

\addtotheorempostheadhook[lemma]{\crefalias{thmlisti}{lemma}}

\addtotheorempostheadhook[remark]{\crefalias{thmlisti}{remark}}

\crefname{lemma}{Lemma}{Lemmas} 
\crefname{conjecture}{Conjecture}{Conjectures}
\crefname{theorem}{Theorem}{Theorems}
\crefname{proposition}{Proposition}{Propositions}
\crefname{definition}{Definition}{Definitions}
\crefname{remark}{Remark}{Remarks}
\crefname{corollary}{Corollary}{Corollaries} 
\crefname{corx}{Corollary}{Corollaries} 
\crefname{thmx}{Theorem}{Theorems}



\newcommand\GL{\mathrm{GL}}

\def\hess{\sqrt{-1}\partial\bar{\partial}}

\def\db{\bar{\partial}}

\DeclareMathOperator{\dvol}{dvol}

\def\C{\mathbb{C}}
\def\bR{\mathbb{R}}

\def\cR{\mathscr{R}}

\def\bC{\mathbb{C}}

\def\cH{\mathcal{H}}

\def\cC{\mathcal{C}}
\def\cD{\mathscr{D}}
\def\cO{\mathcal{O}}

\def\cL{\mathcal{L}}

\author[Y. Deng]{Ya Deng}
\address{CNRS, Institut \'Elie Cartan de Lorraine, Universit\'e de Lorraine, Site de
	Nancy,   54506 Vand\oe uvre-lès-Nancy, France} 
\email{ya.deng@math.cnrs.fr}
\urladdr{https://ydeng.perso.math.cnrs.fr} 

\author[B. Wang]{Botong Wang}
\address{Department of Mathematics, University of Wisconsin-Madison, 480 Lincoln Drive, Madison WI 53706-1388, USA}
\email {wang@math.wisc.edu}
\urladdr{https://people.math.wisc.edu/~bwang274/}
 
\title[ $L^2$-vanishing theorem and a conjecture of Koll\'ar]{$L^2$-vanishing theorem 
	and a conjecture of Koll\'ar}
\alttitle{Théorème d’annulation \(L^2\) et une conjecture de Kollár}
\altkeywords{Conjecture de Kollár, Cohomologie \(L^2\), Théorème d’annulation, Application harmonique, Immeuble de Bruhat–Tits, Conjecture de Shafarevich linéaire, Variation de structure de Hodge (mixte), Groupe fondamental génériquement grand}

\begin{document} 
	
	\begin{abstract}
	In 1995, Koll\'ar conjectured that   a smooth complex projective \(n\)-fold \(X\) with generically large fundamental group has  Euler characteristic \(\chi(X, K_X)\geq 0\).      In this paper, we confirm  the conjecture assuming \(X\) has  linear fundamental group, i.e., there exists a  representation \(\pi_1(X)\to {\rm GL}_N(\mathbb{C})\) with finite kernel.  We deduce the conjecture by proving a stronger \(L^2\) vanishing theorem: for the universal cover \(\widetilde{X}\) of such \(X\), its   \(L^2\)-Dolbeault cohomology \(H_{(2)}^{n,q}(\widetilde{X})=0\) for  \(q\neq 0\).     The main ingredients of the proof are techniques from the linear Shafarevich conjecture along with  some analytic methods.


	\end{abstract}
	\begin{altabstract}
		En 1995, Kollár a conjecturé qu’une variété projective complexe lisse \(X\) de dimension \(n\), dont le groupe fondamental est génériquement  grand, vérifie l’inégalité \( \chi(X, K_X) \geq 0 \). Dans cet article, nous confirmons cette conjecture sous l’hypothèse que le groupe fondamental de \(X\) est linéaire, c’est-à-dire qu’il existe une représentation \( \pi_1(X) \to \mathrm{GL}_N(\mathbb{C}) \) à noyau fini. Notre approche repose sur un théorème d’annulation \(L^2\) plus fort: pour le revêtement universel \( \widetilde{X} \) de \( X \), la cohomologie de Dolbeault \(L^2\) vérifie \( H_{(2)}^{n,q}(\widetilde{X}) = 0 \) pour tout \( q \neq 0 \). Les ingrédients principaux de la démonstration proviennent de techniques liées à la conjecture de Shafarevich linéaire, ainsi que de méthodes analytiques.
	\end{altabstract}
	\maketitle
	\tableofcontents
	\section{Introduction}
	\subsection{Main theorem}
	In the  study of Shafarevich maps, Koll\'ar made the following conjecture (\cite[Conjecture 18.12.1]{Kol95}).
\begin{conj}[Koll\'ar]\label{conj:kollar}
	Let $X$ be a smooth complex projective variety. If $X$ has generically large fundamental group, then $\chi(X,K_X)\geq 0$. 
\end{conj} 
	Following the notations of \cite{Kol95}, we say that $X$ \emph{has generically large fundamental group} (resp. $\varrho:\pi_1(X)\to \GL_N(\bC)$ is a \emph{generically large representation}) if for any irreducible positive-dimensional subvariety $Z$ of $X$ passing through a very general point, the image ${\rm Im}[\pi_1(Z^{\rm norm})\to \pi_1(X)]$ (resp. $\rho({\rm Im}[\pi_1(Z^{\rm norm})\to \pi_1(X)])$) is infinite (see \cite[Definition 2.4]{Kol95} for the precise meaning of ``very general''). Here $Z^{\rm norm}\to Z$ is the normalization of $Z$.     Note that in \cite{CDY22,DY23,DY24}, generically large is called \emph{big}.  
	
In \cite{GL87}, Green and Lazarsfeld  proved Koll\'ar's conjecture when $X$ has maximal Albanese dimension.   In this paper, we will use  methods of $L^2$-cohomology to study  \Cref{conj:kollar}. For a compact  K\"ahler manifold $(X,\omega)$, we denote by $\pi_X:\widetilde{X}\to X$ its universal cover, and  for any non-negative integer $p$ and $q$,  let  $H_{(2)}^{p,q}(\widetilde{X})$ be  the  $L^2$-Dolbeault cohomology group  with respect to the metric $\pi_X^*\omega$. Note that its  definition does not depend on the choice of $\omega$.
	Our main theorem is the following:	\begin{thmx}\label{main:kollar}
		Let $X$ be a  smooth complex projective  $n$-fold. If there exists a generically large representation $\varrho:\pi_1(X)\to \GL_N(\bC)$, then the following statements hold.
		    \begin{thmenum} 
			\item  \label{item:vanishing} $H^{p,0}_{(2)}(\widetilde{X})=0$   for  $0\leq p\leq n-1$ and $H^{n,q}_{(2)}(\widetilde{X})=0$ for  $1\leq q\leq n$. 
			\item \label{item:Euler}The Euler characteristic $\chi(X, K_X)\geq 0$.
			\item  \label{item:strict} If the strict inequality $\chi(X,K_X)>0$ holds, then 
			\begin{enumerate}[label*=(\alph*)]
				\item \label{item:converse2}there exists  {a nontrivial} $L^2$-holomorphic $n$-form on $\widetilde{X}$;
				\item  \label{item:converse} $X$ is of general type. 
			\end{enumerate} 
		\end{thmenum}
	\end{thmx}
	In particular, we prove \Cref{conj:kollar} assuming $\pi_1(X)$ is linear, i.e. there exists a representation $ \pi_1(X)\to \GL_N(\bC)$ with finite kernel. 
	
	The most difficult aspect of proving \cref{main:kollar} is  showing that $H^{p,0}_{(2)}(\widetilde{X})=0$   for  $0\leq p\leq n-1$. We  outline the proof strategy at the beginning of  \Cref{sec:linear}. The  conclusion of \cref{main:kollar} follows from this $L^2$-vanishing result via Atiyah's $L^2$-index theorem. This  implication from \cref{item:vanishing} to \cref{item:Euler} is well known to experts, and similar results have  appeared in various forms in the literature, including \cite{Gr,CD01,Eys99,Eys00,Tak04,JZ}, among others.
	
Let us point out that in the special case where $\varrho$ in \cref{main:kollar} is assumed to be \emph{semisimple}, both \cref{item:vanishing} and \cref{item:Euler} were previously claimed in \cite{JZ}, although certain aspects of their argument appear to require further clarification. Our approach follows a similar overall strategy; however, the novelty in the linear case lies in establishing a partial vanishing theorem  (see \cref{thm:vanishing2}), which requires more refined techniques from harmonic maps into Euclidean buildings  in order to apply the methods developed in the linear Shafarevich conjecture \cite{EKPR12} and thereby prove \cref{item:vanishing}.

	\subsection{Some histories and comparison with previous works}
	In this subsection, $(X,\omega)$ is a compact K\"ahler manifold of dimension $n$ and we denote by $\pi_X:\widetilde{X}\to X$  the universal cover.  In \cite{Gr}, Gromov introduced the notion of \emph{K\"ahler hyperbolicity} in his study of the Hopf conjecture. Recall that  $X$ is K\"ahler hyperbolic if there is a smooth 1-form $\beta$ such that $\pi_X^*\omega=d\beta$ and the norm $|\beta|_{\pi_X^*\omega}$ is bounded from above by a constant.   He then proved the vanishing of $k$-th $L^2$-Betti numbers of $\widetilde{X}$  for $k\neq n$.

	Gromov's idea was later extended by Eyssidieux \cite{Esy97}, in which  more general notions  of \emph{weakly K\"ahler hyperbolicity} are introduced.   Eyssidieux's  work was recently generalized by {Bei}, Claudon, {Diverio}, {Eyssidieux}, and {Trapani} \cite{BDET22,BCDT24}, who studied a {birational} analog of {K\"ahler hyperbolicity}. 
 Also as generalizations of Gromov's work,   Cao-Xavier \cite{CX} and Jost-Zuo \cite{JZ}   introduced the notion of \emph{K\"ahler parabolicity}. They independently observed that the arguments of Gromov concerned with the vanishing of $L^2$-Betti numbers work also under the weaker assumption that  $\beta$ is a smooth 1-form such that $|\beta(x)|_{\pi_X^*\omega}$  has \emph{sub-linear growth}. In other words,   there exists a constant $c>0$ such that  $|\beta(x)|_{\pi_X^*\omega}\leq c(1+d_{\widetilde{X}}(x,x_0))$, where $d_{\widetilde{X}}(x,x_0)$ is the Riemannian distance between $x$ and the base point $x_0$ in $\widetilde{X}$ relative to the metric $\pi^*\omega$.   
	
The formulation and proof of our partial $L^2$-vanishing theorem in  \cref{thm:vanishing2}  are inspired by the aforementioned works, and is introduced specifically for the proof of \cref{main:kollar}.  As a result, it may appear technically involved.  
	
Another key component of the proof of \cref{main:kollar} draws on techniques developed in the context of the reductive and linear Shafarevich conjectures, as explored  in \cite{DY23, Eys04, EKPR12}. The Shafarevich conjecture predicts that the universal cover of a smooth complex projective variety is holomorphically convex. This was proved by Eyssidieux, Katzarkov, Pantev, and Ramachandran \cite{Eys04, EKPR12} for smooth projective varieties with linear fundamental groups. More recently, the conjecture was extended by Yamanoi and the first author in \cite{DY23} to projective normal varieties with reductive fundamental groups, and further studied in \cite{DY24} for projective normal varieties with fundamental groups in positive characteristic.

	\subsection{Notation and Convention}
	\begin{itemize}
	\item All varieties in this paper are defined over $\C$.
		\item Let $(X.\omega)$ be a K\"ahler manifold. We denote by $\pi_X:\widetilde{X}\to X$ the universal cover of $X$.  
		\item Let $(X,\omega)$ be a compact K\"ahler manifold. Unless otherwise specified, $d_{\widetilde{X}}(x, x_0)$ stands for the Riemannian distance between $x$ and the base point $x_0$ in $\widetilde{X}$ with respect to the metric $\pi_X^*\omega$. 
		\item For a complex space $Z$, $Z^{\rm norm}$ denotes its normalization, and $Z^{\rm reg}$ denotes its regular locus. 
		\item We use the standard abbreviations VHS and VMHS for \emph{variation of Hodge structure} and \emph{variation of mixed Hodge structure} respectively.
		\item By convention, a closed positive $(1,1)$-current $T$ on a complex manifold is   \emph{semi-positive}. 
		\item Plurisubharmonic functions  are  abbreviated as psh functions. 
		\item A positive closed $(1,1)$-current $T$  has \emph{continuous potential} if locally we have $T=\hess\psi$ with $\psi$  a continuous psh function. 
		\item By convention, a real \emph{positive} \( (1,1) \)-form or current is understood to be \emph{semi-positive} in the usual sense. 
	\end{itemize}
	\subsection{Acknowledgment}    We would like to express our gratitude to Patrick Brosnan, Junyan Cao, Simone Diverio, Philippe Eyssidieux, Mihai P\u{a}un, Pierre Py, Christian Schnell,  Xu Wang, and Mingchen Xia for their inspiring discussions. We extend special thanks to Chikako Mese for her invaluable discussion  and expertise on harmonic maps into Euclidean buildings, and  to Yuan Liu and Xueyuan Wan for reading the first version of the paper and their very helpful remarks.   YD acknowledges the partial support of the French Agence Nationale de la Recherche (ANR) under reference ANR-21-CE40-0010. BW thanks Peking University and BICMR for the generous hospitality and support during the writing of this paper. He is partially supported by a Simons fellowship.
	
	\section{A partial $L^2$-vanishing theorem}
 \label{sec:kp}

	Before proving a key $L^2$-vanishing theorem, we first recall the definition of $L^2$-cohomology (cf. \cite[Definition 12.3]{DIP02}).
	\begin{definition}[$L^2$-cohomology]\label{defn:L2} 
		Let $(Y,\omega)$ be a complete K\"ahler manifold.  Let   $L_{(2)}^{p,q}(Y)$  be the space of $L^2$-integrable   $(p,q)$-forms with respect to the metric $\omega$. A section $u$ is said to be in  ${\rm Dom}\, \db$  if  $\db u$  calculated in the sense of distributions is still in $L^2$.  Then  the $L^2$-Dolbeault cohomology is defined as
		\[ H^{p,q}_{(2)}(Y)= {{\rm ker}\, \db}\big/\,{\overline{{\rm Im}\, \db \cap {\rm Dom}\, \db}}.
		\]  
		If $(X,\omega)$ is a compact K\"ahler manifold, then 	$H^{p,q}_{(2)}(\widetilde{X})$ denotes the $L^2$-cohomology with respect to the metric $\pi_X^*\omega$. It is well known that this cohomology group is independent of the choice of Kähler metric $\omega$. 
	\end{definition} 
	
	Let us state and prove our partial  $L^2$-vanishing theorem. 
	\begin{theorem} \label{thm:vanishing2} 
 Let   $(X,\omega)$ be a compact 	 K\"ahler $n$-fold.  Let $f:X\to Y$ be a proper surjective holomorphic map with connected fibers over a compact  K\"ahler normal space $Y$ of dimension $m$. Denote by $\tilde{f}:\widetilde{X}\to \widetilde{Y}$ the lift of $f$ to  universal covers.   Assume that  there exists   a  1-form  $\beta$ on $\widetilde{X}$  with $L_{\rm loc}^1$-coefficients  satisfying the following properties. 
 	\begin{enumerate}[label*=(\rm \alph*)]
 	\item \label{item:smooth} The 1-form $\beta$ is smooth  on an open subset   of $\widetilde{X}$ whose complement has zero Lebesgue measure, and  there is a constant $C>0$ such that  \begin{align}\label{item:dlinear}
 	 |\beta(x)|_{\pi_X^*\omega}\leq_{\rm a.e.} C(d_{\widetilde{X}}(x,x_0)+1)
 \end{align}
 where $x_0$ is a base point of $\widetilde{X}$. 
 	\item\label{item:bounded}   The current \( d\beta \) is a real positive \( (1,1) \)-current satisfying
 	\begin{align}\label{eq:mutual}
 		d\beta \geq \pi_X^* f^* T_Y,
 	\end{align}
 	where \( T_Y \) is a closed positive \( (1,1) \)-current on \( Y \). Moreover, there exists a non-empty analytic open subset \( Y_0\subset Y \) such that \( T_Y|_{Y_0} \) is smooth and strictly positive.
 \end{enumerate} 
 Then  \begin{thmlist}
 	 \item\label{item:partial}  For any $p\in \{0,\ldots,m-1\}$, we have  $H^{p,0}_{(2)}(\widetilde{X})=0$. 
 	\end{thmlist}
Assume that there exists a non-zero class $\alpha \in H^{p,0}_{(2)}(\widetilde{X})$ for some $p\in \{m,\ldots,n-1\}$. Let $Y^\circ$ be a non-empty open subset of $Y^{\mathrm{reg}} \cap Y_0$ over which $f$ is a proper submersion, and such that the restriction $T_Y|_{Y^\circ}$ is a K\"ahler form.  Then
 	\begin{thmlist} [resume]
 	 \item\label{item:long}  we have   \begin{align}\label{eq:tensor}
 	 	 \alpha|_{\widetilde{X}^\circ}\in H^0(\widetilde{X}^\circ, \tilde{f}^*\Omega_{\widetilde{Y}^\circ}^{m}\otimes \Omega_{\widetilde{X}^\circ}^{p-m}), 
 	 \end{align}
where $\widetilde{Y}^\circ:= \pi_Y^{-1}(Y^\circ)$ and $\widetilde{X}^\circ:=\tilde{f}^{-1}(\widetilde{Y}^\circ)$.
\item \label{item:d-closed} For any $y\in \widetilde{Y}^\circ$,  let $\alpha_y$ be the holomorphic $(p-m)$-form on $\tilde{f}^{-1}(y)$ induced by $\alpha$   under the isomorphism $$\tilde{f}^*\Omega_{\widetilde{Y}^\circ}^{m}\otimes \Omega_{\widetilde{X}^\circ}^{p-m}|_{\tilde{f}^{-1}(y)}\simeq \Omega_{\tilde{f}^{-1}(y)}^{p-m}.$$ 
 	Then  $
 	 \alpha|_{\tilde{f}^{-1}(y)}
 	 $ is   $d$-closed.
 \end{thmlist}
	
	\end{theorem}
	\begin{proof}
  To lighten the notation, we write  $\omega$   instead of   $\pi_X^*\omega$ abusively.  
   
   \noindent {\bf Step 1. }
Since the sectional curvature of the complete K\"ahler manifold $(\widetilde{X},\pi_X^*\omega)$ is uniformly bounded,  by a result of W. Shi (see e.g. \cite[Theorem 1.2]{Hua19}),  there exists a constant $C>0$ and a  smooth exhausting function $r$   on $\widetilde{X}$ such that 
$$
d_{\widetilde{X}}(x,x_0)+1 \leq r(x)\leq d_{\widetilde{X}}(x,x_0)+C,
$$ 
and $|d r|_{\omega}(x) \leq C$  for all $x \in \widetilde{X}$.   Hence by \eqref{item:dlinear}, we have 
\begin{align}\label{eq:betanorm}
	 |\beta|_\omega(x)\leq_{\rm a.e.} Cr(x).
\end{align} 	Let $\varrho: \mathbb{R} \rightarrow \mathbb{R}$ be a smooth function with $0 \leq \varrho \leq 1$ such that
		$$
\varrho(t)= \begin{cases}1, & \text { if } t \leq 0; \\ 0, & \text { if } t \geq 1.\end{cases}
		$$
		We consider the compactly supported function
		\begin{align}\label{eq:support}
			f_j(x)=\varrho(r(x)-{j}+1),
		\end{align} 
		where $j$ is a positive integer.  Then  $${\rm Supp}(f_j)\subset \{x\in \widetilde{X}\mid   r(x)\leq  {j} \},$$ 
		$$
		df_j(x)=\varrho'(r(x)-j+1) dr.
		$$ 
	Then  there exists some constant $c_1>0$ such that 
		\begin{align} \label{eq:uniform}
 |df_j(x)|_{\omega}\leq c_1. 
		\end{align}
		for any $x\in \widetilde{X}$. 
		
	
		
	  Let $\alpha$ be a  holomorphic $(p,0)$-form which is $L^2$ with respect to $\omega$ with $0\leq p\leq n-1$. Then 
	\begin{align}\label{eq:finite}
 \sqrt{-1}^{p^2}\int_{\widetilde{X}}\alpha\wedge \bar{\alpha}\wedge \omega^{n-p}	=	\frac{(n-p)!}{n!}\int_{\widetilde{X}  }    |\alpha|_{\omega}\omega^n <\infty.
		\end{align}

		Denote by 
		$$
		B_j:=\{x\in \widetilde{X}\mid   j-1\leq r(x)\leq  j \}.
		$$  
	In what follows, 	for any smooth form $\gamma$ on $\widetilde{X}$, we denote by $|\gamma|$ its norm with respect to the metric $\omega$.  	Denote $\dvol:=\frac{\omega^n}{n!}$.    
Recall that  $|\beta(x)| \leq_{\rm a.e.} Cr(x)$ for some constant $C>0$. Hence by $\operatorname{supp}( df_j) \subset {B_{j}} $,  one has
		\begin{align} \nonumber
			\left|  \int_{\widetilde{X} }  df_j\wedge \beta\wedge\alpha\wedge \bar{\alpha}\wedge \omega^{n-p-1}  \right| 
			&\leq   \int_{  B_j} \left|  df_j\wedge \beta\wedge\alpha\wedge \bar{\alpha}\wedge \omega^{n-p-1}\right| \dvol   \\ \nonumber
			&\leq \int_{B_j}     |d f_j| \;  |\beta| \; | {\alpha}|^2 \; |\omega^{n-p-1}|   \dvol\\\label{eq:small}
			&\stackrel{\eqref{eq:uniform}\& \eqref{eq:betanorm} }{\leq} c_1Cj\int_{B_j}     | {\alpha}|^2   \dvol.  
		\end{align}
		
		\begin{claim}\label{claim:converge}
			There exists a subsequence $\left\{j_i\right\}_{i \geq 1}$ such that
			\begin{align}\label{eq:crucial}
				\lim _{i \rightarrow \infty} j_i  \int_{B_{j_i}}|\alpha(x)|^2 \dvol=0 .
			\end{align}  
		\end{claim}	 
		\begin{proof}
			If not, then there exists a positive constant $c'$ and a positive integer $n_0$ such that 
			$$
		j\int_{B_{j} }|\alpha(x)|^2 \dvol \geq c'>0
			$$
			for any $j\geq n_0$.
			This yields
			$$
			\begin{aligned}
			 \int_{\widetilde{X}}|\alpha(x)|^2 \dvol & \geq \sum_{j=n_0}^{\infty} \int_{B_{j}}  |\alpha(x)|^2 \dvol \\
				& \geq c' \sum_{j=n_0}^{+\infty} \frac{1}{j}=+\infty,
			\end{aligned}
			$$
		thereby yielding a contradiction, as $\alpha$ is assumed to be an $L^2$-form with respect to $\omega$.
		\end{proof}	
			\Cref{claim:converge} implies that there exists a subsequence $\left\{j_i\right\}_{i \geq 1}$ for which \eqref{eq:crucial} holds.  By the Stokes formula, we have
		\begin{align}\label{eq:essential} 
			\sqrt{-1}^{p^2}	\int_{\widetilde{X} } f_j  d\beta \wedge \alpha\wedge \bar{\alpha}\wedge \omega^{n-p-1}
			= \sqrt{-1}^{p^2}	\int_{\widetilde{X} }  (-df_j\wedge \beta)\wedge\alpha\wedge \bar{\alpha}\wedge \omega^{n-p-1}    
		\end{align} 
	 	By \cite[Chapter III, Example 1.2]{Dembook}, $\sqrt{-1}^{p^2}	  \alpha\wedge \bar{\alpha}$ is a positive $(p,p)$-form in the sense of \cite[Chapter III, Definition 1.1]{Dembook}. 	 \cite[Chapter III, Corollary 1.9 \& Proposition 1.11]{Dembook} imply that $\sqrt{-1}^{p^2}\alpha\wedge \bar{\alpha}\wedge \omega^{n-p-1}$ is a positive $(n-1,n-1)$-form, hence \emph{strongly positive} (not strictly positive!) in the sense of  \cite[Chapter III, Definition 1.1]{Dembook}. 	  Since $d\beta$ is assumed to be a positive current, it follows that
		 $$
 	\sqrt{-1}^{p^2}	   d\beta \wedge \alpha\wedge \bar{\alpha}\wedge \omega^{n-p-1}
		 $$
		 is a  positive measure on $\widetilde{X}$.   \eqref{eq:small}    and \eqref{eq:essential}   imply that
		 $$ 
	\lim\limits_{i\to\infty}	\int_{\widetilde{X} } f_{j_i}	\sqrt{-1}^{p^2}	 d\beta \wedge \alpha\wedge \bar{\alpha}\wedge \omega^{n-p-1}=0.
		 $$ 
Recall that $0 \leq f_j \leq 1$,  and $ f_j(x) =1$ for any $x$ satisfying  $d_{\widetilde{X}}(x,x_0)\leq j-1-C$. The above equality together with the monotone convergence theorem imply that 
		\begin{align*} 
		\int_{\widetilde{X} }  	\sqrt{-1}^{p^2}	 d\beta \wedge \alpha\wedge \bar{\alpha}\wedge \omega^{n-p-1}=0.
		\end{align*}  
		It follows \eqref{eq:mutual} that 
\begin{align}\label{eq:pointwise}
\int_{\widetilde{X}}  	\sqrt{-1}^{p^2}	  \alpha\wedge \bar{\alpha} \wedge \omega^{n-p-1}\wedge \pi_X^*f^*T_Y=0. 
\end{align}

\medspace

\noindent {\bf Step 2.}  
We abusively write \( T_Y \) to denote \( \pi_Y^*T_Y \). Let \( y \in \widetilde{Y}^\circ \) be any point, and let \( x \in \widetilde{X} \) be any point in the fiber \( \tilde{f}^{-1}(y) \). Since \( f \) is smooth over \( Y^\circ \), we can find coordinate neighborhoods \( (U; z_1, \ldots, z_n) \) centered at \( x \) and \( (V; w_1, \ldots, w_m) \) centered at \( y \), such that 
\[
\tilde{f}(z_1, \ldots, z_n) = (z_1, \ldots, z_m).
\]
Since \( T_Y|_{Y^\circ} \) is a Kähler form on \( Y^\circ \), after applying a unitary change of coordinates, we may assume that 
\[
\omega(x) = \sqrt{-1} \sum_{j=1}^n dz_j \wedge d\bar{z}_j \quad \text{and} \quad T_Y(y) = \sum_{i=1}^m \lambda_i \sqrt{-1} dw_i \wedge d\bar{w}_i,
\]
where each \( \lambda_i \in \mathbb{R}_{>0} \).

	For any subset \( I = \{i_1, \ldots, i_p\} \subset \{1, \ldots, n\} \), we denote \( dz_I := dz_{i_1} \wedge \cdots \wedge dz_{i_p} \). Let \( \alpha \in H^{p,0}_{(2)}(\widetilde{X}) \). If we write
	\[
	\alpha|_U = \sum_{|I| = p} \alpha_I \, dz_I, \quad \text{with } \alpha_I \in \mathcal{O}(U),
	\]
	then we have
	\[
	\sqrt{-1}^{p^2} \, \alpha \wedge \bar{\alpha} \wedge \tilde{f}^*T_Y \wedge \omega^{n-p-1}(x)
	= \frac{(n - p - 1)!}{n!}\sum_{|I|=p} \sum_{j \in \{1, \ldots, m\} \setminus I} \lambda_j |\alpha_I(x)|^2 \, \omega(x)^n.
	\]
	By \eqref{eq:pointwise}, it follows that \( \{1, \ldots, m\} \subset I \) for any multi-index \( I \) with \( \alpha_I(x) \neq 0 \). Since \( x \) is arbitrary in the non-empty analytic open subset \( \widetilde{X}^\circ \subset \widetilde{X} \), we conclude that \( \alpha(x) = 0 \) for all \( x \in \widetilde{X}^\circ \) if $p<m$. As \( \alpha \) is a holomorphic \((p,0)\)-form, it follows that \( \alpha \equiv 0 \). This establishes \Cref{item:partial}.
	
	Moreover, if \( n > m \) and \( p \geq m \), then this also proves \eqref{eq:tensor}. Hence, \Cref{item:long} is proved.

 \medspace
 
 \noindent {\bf Step 3.}   
 Let us now prove \Cref{item:d-closed}.  
 For any point \( y \in \widetilde{Y}^\circ \), denote the fiber over \( y \) by \( F_y := \tilde{f}^{-1}(y) \).  
 Every nontrivial \( m \)-form \( \eta \in \Omega^m_{\widetilde{Y}^\circ,y} \) induces a \emph{unique} holomorphic \((p - m)\)-form \( \alpha_y \in H^0(F_y, \Omega^{p-m}_{F_y}) \) such that
 \[
 \alpha = \alpha_y \wedge \tilde{f}^*\eta.
 \]
 This form \( \alpha_y \) can be written explicitly. Using the coordinate charts \( (U; z_1, \ldots, z_{n+m}) \) and \( (V; w_1, \ldots, w_m) \) introduced in Step 2, we may write
 \[
 \eta = \lambda \, dw_1 \wedge \cdots \wedge dw_m,
 \]
 for some \( \lambda \in \mathbb{C}^* \) (note that \( \lambda \) depends on the choice of coordinates).
 
 For any multi-index \( I \supset \{1, \ldots, m\} \), denote by \( \tilde{I} := I \setminus \{1, \ldots, m\} \).  
 Then, by \Cref{item:long}, we have  
\begin{align}\label{eq:express2}
	 \alpha|_U=\sum_{|I|=p,I\supset\{1,\ldots,m\}}\alpha_Idz_{\tilde{I}}\wedge dz_1\wedge \cdots\wedge dz_m.
\end{align} 
Then the holomorphic \((p - m)\)-form \( \alpha_y \) on \( F_y \cap U \) is defined by
\begin{align}\label{eq:express}
	\alpha_y := \lambda^{-1} \cdot \sum_{|I| = p,\; I \supset \{1, \ldots, m\}} \alpha_I|_{F_y} \, dz_{\tilde{I}}.
\end{align}
Since \( \alpha_y \wedge \tilde{f}^*\eta = \alpha \), it follows that the definition of \( \alpha_y \) is independent of the choice of coordinates. Therefore, as we vary the coordinate open subsets, the local forms glue together to define a global holomorphic \((p - m)\)-form \( \alpha_y \) on \( F_y \).

Now, since \( \alpha \) is \( d \)-closed by assumption, the expression for \( \alpha \) implies
\begin{align}\label{eq:together}
	0 = d\alpha = \lambda^{-1} \cdot \sum_{\substack{|I| = p\\ I \supset \{1, \ldots, m\}}} \sum_{j = m + 1}^{n} \frac{\partial \alpha_I}{\partial z_j} \, dz_j \wedge dz_{\tilde{I}} \wedge dz_1 \wedge \cdots \wedge dz_m.
\end{align}
This yields
\[
d\alpha_y = \lambda^{-1} \cdot \sum_{\substack{|I| = p\\ I \supset \{1, \ldots, m\}}} \sum_{j = m + 1}^{n} \left( \frac{\partial \alpha_I}{\partial z_j} \bigg|_{F_y} \right) dz_j \wedge dz_{\tilde{I}} = 0.
\]
This proves \Cref{item:d-closed}, and completes the proof of the theorem. 
	\end{proof}

	We will need the following consequence of \Cref{thm:vanishing2} in the proof of \Cref{main:kollar}. 
	\begin{corollary}\label{prop:abelian}
		Let \( (X, \omega) \) be a compact Kähler $n$-fold, and let \( f: X \to A \) be a holomorphic map to an abelian variety \( A \) such that \( \dim f(X) = n \). Let \( \widetilde{X}_1 \) be a connected component of the fiber product \( X \times_A \widetilde{A} \), where \( \widetilde{A} \to A \) denotes the universal cover of \( A \). Then, for any infinite Galois cover \( \pi':\widetilde{X}' \to X \) that factors through \( \widetilde{X}_1 \), we have
		\[
		H^{p,0}_{(2)}(\widetilde{X}') = 0 \quad \text{for all } p \in \{0, \ldots, n - 1\}.
		\]
		\end{corollary}
		
		\begin{proof}
	 We choose global linear coordinates \( (z_1, \ldots, z_m) \) on \( \widetilde{A} \) such that the standard Kähler form
		\[
		\omega_{\widetilde{A}} := \sqrt{-1} \sum_{i=1}^{m} dz_i \wedge d\bar{z}_i
		\]
		descends to a Kähler form \( \omega_A \) on \( A \). Since \( \widetilde{X}' \) dominates \( \widetilde{X}_1 \), the map \( f: X \to A \) lifts to a holomorphic map \( g: \widetilde{X}' \to \widetilde{A} \). 
		Set \( g_i := z_i \circ g \) for each \( i \); this defines a holomorphic function on \( \widetilde{X}' \).  
		
		We consider the smooth real \((1,0)\)-form
		\[
		\beta := g^*\left( \sqrt{-1} \, \bar{\partial} \sum_{i=1}^{m} |z_i|^2 \right) 
		= \sqrt{-1} \sum_{i=1}^{m} g_i(x) \, \overline{\partial g_i(x)}.
		\]
		Since each \( dz_i \) descends to a holomorphic 1-form on \( A \), it follows that each \( \partial g_i(x) \) descends to a holomorphic 1-form on \( X \). Therefore, there exists a constant \( C > 0 \) such that
		\[
		|\partial g_i(x)|_{\pi'^* \omega} \leq C
		\quad \text{for all } x \in \widetilde{X}' \text{ and } i = 1, \ldots, m.
		\]
		In particular, the map \( g \) is Lipschitz. Fix a base point \( x_0 \in \widetilde{X}' \). Then there exists a constant \( C' > 0 \) such that
		\[
		|g_i(x)| \leq C' \left( d_{\widetilde{X}'}(x, x_0) + 1 \right)
		\quad \text{for all } x \in \widetilde{X}'.
		\]
		It follows that
		\[
		|\beta(x)|_{\pi'^* \omega} \leq m C C' \left( d_{\widetilde{X}'}(x, x_0) + 1 \right).
		\]
		
		Note that
		\[
		d\beta = (\pi')^* f^* \omega_A,
		\]
		where \( f^* \omega_A \) is a closed semi-positive \((1,1)\)-form on \( X \), which is strictly positive on a non-empty Zariski open subset where \( f \) is immersive.
		
		By \Cref{item:partial}, this implies the desired \( L^2 \)-vanishing result.
	 	\end{proof}

	\section{Constructing  1-forms via harmonic maps to Euclidean buildings}\label{sec:1-form}  
	Let $K$ be a non-archimedean local field of characteristic zero and let $G$ be a reductive algebraic group defined over $K$. There exists a Euclidean building associated with $G$, which is  the enlarged Bruhat-Tits building and denoted by $\Delta(G)$.  We refer the readers to \cite{KP23,Rou23} for the definition and properties of Bruhat-Tits buildings.

	Let $(V,W,\Phi)$ be the root system associated with $\Delta(G)$. It means that $V$ is a real Euclidean space endowed with a Euclidean metric and $W$ is an affine Weyl group acting on $V$. Namely, $W$ is a semidirect product   $T\rtimes W^v$,  where $W^v$ is the \emph{vectorial Weyl group}, which is a finite group generated by reflections on $V$, and $T$ is a translation group of $V$.  Here $\Phi$ is the root system of $V$.  It is a finite set of $V^*\backslash \{0\}$ such that $W^v$ acts on $\Phi$ as a permutation.  Moreover, $\Phi$ generates $V^*$.  From the reflection hyperplanes of $W$ we obtain a decomposition of $V$ into facets. Let $\cH$ be set of  hyperplanes of $V$ defined by $w\in W$.  The maximal facets, called \emph{chambers}, are the open connected components of $V\backslash \cup_{H\in \cH}$.

For any apartment $A$ in $ \Delta(G)$,  there exists an isomorphism $i_A:A\to V$, which is  called a chart. For two charts $i_{A_1}:A_1\to V$ and $i_{A_2}:A_2\to V$, if $A_1\cap A_2\neq\varnothing$,   it satisfies the following properties:
	\begin{enumerate}[label=(\alph*)]
		\item $Y:=i_{A_2}(i_{A_1}^{-1}(V))$ is convex.
		\item \label{item:weyl}There is an element $w\in W$ such that $w\circ i_{A_1}|_{A_1\cap A_2}=i_{A_2}|_{A_1\cap A_2}$.
	\end{enumerate} 
	The charts allow us to map facets into $\Delta(G)$ and their images are
	also called facets. The  axioms guarantee that these
	notions are chart independent. 

Let $X$ be a compact K\"ahler manifold and 	let $\varrho:\pi_1(X)\to G(K)$ be a Zariski dense representation.   	By the work of Gromov-Schoen \cite{GS92} (see also \cite{BDDM} for the quasi-projective case), there exists a $\varrho$-equivariant harmonic mapping $u:\widetilde{X}\to \Delta(G)$.  Such $u$ is moreover pluriharmonic.  	We denote by $R(u)$ the regular set of harmonic map $u$. That is, for any $x\in R(u)$, there exists  an open subset $\Omega_{x}$ containing $x$ such that $u(\Omega_{x})\subset A$ for some apartment $A$.   
Since $G(K)$ acts transitively on the apartments of $\Delta(G)$,    we know that $R(u)$ is the pullback of an open subset $X'$ of $X$. By \cite{GS92}, $X\backslash X'$  has \emph{Hausdorff codimension} at least two.  

 We fix  an orthogonal coordinates $(x_1,\ldots,x_N)$ for $V$.  Define   smooth real functions on $\Omega_x$ by setting
\begin{align}\label{eq:ui}
 u_{A,i}:=x_i\circ i_A\circ u.
\end{align}   The pluriharmonicity of $u$ implies that $\hess u_{A,i}=0$. We consider a smooth real semi-positive $(1,1)$-form on $\Omega_x$ defined by 
\begin{align}\label{eq:form}
   \sqrt{-1}\sum_{i=1}^{N}\partial u_{A,i}\circ \db u_{A,i}. 
 \end{align}
By \cite[\S 3.3.2]{Eys04},  such real $(1,1)$-form does not depend on the choice of $A$, and is invariant under $\pi_1(X)$-action. Therefore, it descends to a smooth real closed semi-positive $(1,1)$-form on $X'$.  
 It is shown in \cite{Eys04} that it extends to a  positive closed $(1,1)$-current $T_{\varrho}$  on $X$ with continuous potential.  
\begin{definition}[Canonical current]\label{def:canonical}
	The above closed positive $(1,1)$-current 	$T_\varrho$ on $X$ is called the \emph{canonical current} of $\varrho$.   
\end{definition}  
\begin{remark}
	Although the $\varrho$-equivariant harmonic map $u$ may not be unique, the unicity result in \cite{DM24} implies that the  $(1,1)$-form \eqref{eq:form} is independent of the choice of $u$. Since the current $T_\varrho$ is defined as the minimal extension of this $(1,1)$-form, it follows that $T_\varrho$ is uniquely determined and independent of the choice of $u$.
\end{remark}
By \cite{Eys04,CDY22,DY23}, one has  a proper fibration  $s:X\to S_{\varrho}$ associated with $\varrho$,  which is called \emph{Katzarkov-Eyssidieux reduction map}. It has the following properties. 
\begin{proposition}[\cite{Eys04,CDY22}]\label{thm:KE}
	Let $X$ be a smooth projective variety and let $\rho:\pi_1(X)\to G(K)$ be a  Zariski dense representation where $G$ is a reductive group over  a non-archimedean local field $K$.  Then there exists a proper morphism $s_\rho:X\to S_\rho$ (so-called Katzarkov-Eyssidieux reduction map) onto a normal projective variety with connected fibers  such that  for any irreducible closed subvariety $Z\subset X$, the following properties are equivalent: 
	\begin{enumerate}
		\item\label{item:KZ1} $\rho({\rm Im}[\pi_1(Z_{\rm norm})\to \pi_1(X)])$ is bounded;
		\item \label{item:KZ3} $\rho({\rm Im}[\pi_1(Z)\to \pi_1(X)])$ is bounded;
		\item \label{item:KZ2}  $s_\rho(Z)$ is a point.
	\end{enumerate}  
	Moreover, there exists a $(1,1)$-current $T_\varrho'$ with continuous potential on $S_\varrho$ such that $s_\varrho^*T_\varrho'=T_\varrho$. \qed
\end{proposition}

	\begin{proposition}\label{lem:equal}
Let $x_0$ be a fixed base point in $\widetilde{X}$. For the canonical current $T_\varrho$ defined in \cref{def:canonical}, we have
	\begin{align}\label{eq:equal} 
 \hess d^2_{\Delta(G)}(u(x),u(x_0)) \geq\pi_X^*T_\varrho,
	\end{align} 
	where $d_{\Delta(G)}(\bullet,\bullet)$ is the distance function on $\Delta(G)$. 
\end{proposition}
\begin{proof}
We assume that $x_0\in R(u)$.  Let $\omega$ be a K\"ahler metric on $X$. 	For any $x\in R(u)$ where $R(u)$ is the regular set of the harmonic map $u$, there exists  an open subset $\Omega_{x}$ containing $x$ such that $u(\Omega_{x})\subset A$ for some apartment $A$. 
	
Note that $ {d}_{\Delta(G)}$ is   $G(K)$-invariant.  Let $\iota:C\to \Omega_x$ be any holomorphic curve.  Then by \cite[Proposition 2.2 in p. 191]{GS92},  we have
	$$
	\Delta d^2(u\circ \iota (x), u(x_0) )\geq |\nabla u\circ \iota|^2,
	$$
	where the  $|\nabla u\circ \iota|^2$ is the norm defined with respect to  $d_{\Delta(G)}$ and $\omega$.  It follows that 
	$$
	|\nabla u\circ \iota|^2 \iota^*\omega= \iota^*(\sqrt{-1}\sum_{i=1}^{N}\partial u_{A,i}\wedge\db u_{A,i})
	$$
	where $u_{A,i}$ is defined in \eqref{eq:ui}. 
	Therefore, over $\Omega_x$ we have
	$$\hess  {d}_{\Delta(G)}^2(u(x), u(x_0) )\geq \sqrt{-1}\sum_{i=1}^{m}\partial u_{A,i}\wedge\db u_{A,i}.$$
By \cref{def:canonical}, we have  \eqref{eq:equal} over the whole $R(u)$. 
	\begin{claim}
\eqref{eq:equal} holds over the whole $\widetilde{X}$.
	\end{claim}
	\begin{proof}
		 Since $u$ is  Lipschitz, it follows that   $d^2_{\Delta(G)}(u(x),u(x_0))$ is a continuous function on $\widetilde{X}$. Recall that $T_\varrho$ is a positive closed $(1,1)$-current with continuous potential. This implies that for any point $x\in \widetilde{X}$, there exist a neighborhood $\Omega_x$ and a continuous function $\phi$ on $\Omega_x$ such that 
		 $$
\hess d^2_{\Delta(G)}(u(x),u(x_0))- \pi_X^*T_\varrho=\hess \phi 
		 $$
		 on $\Omega_x$. Note that $\hess\phi\geq 0$ over $\Omega_x\cap R(u)$. Since the complement of $\Omega_x\cap R(u)$ has Hausdorff codimension at least two in $\Omega_x$, we apply the extension theorem in \cite[Theorem 3.1(i)]{Shi72} to conclude that there is a  psh function $\phi'$ on $\Omega_x$ such that $\phi'|_{\Omega_x\cap R(u)}=\phi|_{\Omega_x\cap R(u)}$.    Since  $\phi$ is continuous,  and $\phi'$ is upper semi-continuous, it follows that   $\phi=\phi'$ on $\Omega_x$.  This shows that \eqref{eq:equal} holds over $\Omega_x$, hence over the whole $\widetilde{X}$.  The claim is proved. 
	\end{proof}
		The proposition is proved.   
\end{proof}	  

Let us denote by
\begin{align}\label{eq:defbeta}
	\beta_\varrho(x):=\sqrt{-1}\db d^2_{\Delta(G)}(u(x),u(x_0)).
\end{align}
Since $d^2_{\Delta(G)}(u(x),u(x_0))$ is a psh function on $\widetilde{X}$ by \cref{lem:equal}, by \cite[Theorem 1.46]{GZ17}, $\beta_\varrho$ has $L_{\rm loc}^1$-coefficients. 
 \begin{lemma}\label{lem:distance}
 There exists a dense open subset $\widetilde{X}^\circ $ of $\widetilde{X}$ whose complement has zero Lebesgue measure such that $\beta_\varrho$ is smooth. Moreover, there exists  	a number $c>0$ such that for any $x\in \widetilde{X}^\circ$, we have
 	\begin{align*}
 		|\beta_\varrho(x)|_{\pi_X^*\omega} \leq  c(1+d_{\widetilde{X}}(x, x_0)).
 	\end{align*}  Here $\omega$ is a K\"ahler metric on $X$.
 \end{lemma}
 \begin{proof}
 Let $\widetilde{X}^\circ$ be the set of points $x$ in $\widetilde{X}$ such that there exists an open neighborhood $\Omega_x$ containing $x$ such that $u(\Omega_x)$ is contained in the closure of a chamber $\cC$ of the building $\Delta(G)$.  By \cite[Proposition 4.15]{DM24}, $\widetilde{X}^\circ$ is a dense open subset such that $\widetilde{X}\backslash \widetilde{X}^\circ$  has zero Lebesgue measure. Note that there exists an apartment $A$ containing both $\overline{\cC}$ and $x_0$. It follows that for any $y\in \Omega_{x}$, we have
 \begin{align}\label{eq:distanced}
 	d^2_{\Delta(G)}(u(y),u(x_0))=\sum_{i=1}^{N}|u_{A,i}(y)-u_{A,i}(x_0)|^2.
 \end{align}
 Then  by \eqref{eq:distanced} and \eqref{eq:defbeta}, one has
 	$$
 \beta_\varrho(y)= \sqrt{-1} \sum_{i=1}^{N} (u_{A,i}(y)-u_{A,i}(x_0))\db u_{A,i}(y).
 	$$
 	Since $u$ is   Lipschitz and $\varrho$-equivariant,  it follows that there exists a uniform  constant $c_1>0$ such that for any $y\in \widetilde{X}$, we have $$d_{\Delta(G)}(u(y),u(x_0))\leq  c_1(1+d_{\widetilde{X}}(y, x_0)).$$  
 Note that for any $y\in \Omega_x$, we have
 	$$
 	|u_{A,i}(y)-u_{A,i}(x_0)|\leq  d_{\Delta(G)}(u(y),u(x_0)).
 	$$ 
 	On the other hand, by the Lipschitz condition of $u$, there exists another uniform constant $c_2>0$ such that for any $y\in \Omega_x$, we have
 	$$
 	\sum_{i=1}^{N}|\db u_{A,i}(y)|_{\pi^*\omega}\leq  c_2.
 	$$
 	In conclusion, we have
 	\begin{align*}
 		|\beta_\varrho(y)|_{\pi_X^*\omega} \leq    c_1c_2(1+d_{\widetilde{X}}(y, x_0))\quad \mbox{for any }y\in \Omega_x.
 	\end{align*} 
 	Since $x$ is any point in $\widetilde{X}^\circ$, the above inequality holds for any $x\in \widetilde{X}^\circ$.  The lemma is proved. 
 \end{proof}

	\section{$L^2$-vanishing theorem and generically large  local systems}\label{sec:linear}
 In this section we will prove \cref{main:kollar}.  In \Cref{subsec:semisimple},   we address the case where $\varrho$ is semisimple, utilizing techniques from the proof of the reductive Shafarevich conjecture in   \cite{Eys04,DY23}.  The desired 1-form $\beta$, as required in \Cref{thm:vanishing2},  arises from   1-forms $\beta_\tau$ associated with $\tau:\pi_1(X)\to \GL_N(K)$ defined in \eqref{eq:defbeta}, where $K$ is a non-archimedean local field.   We then reduce the proof   to \Cref{item:partial}.

 For the proof of the general cases of \cref{main:kollar}, we will apply techniques from the proof of the linear Shafarevich conjecture in \cite{EKPR12}.   
 Using similar techniques as in the semi-simple case,   we first  construct a suitable fibration $f:X\to Y$ (the reductive Shafarevich morphism ${\rm sh}_M^0:X\to {\rm Sh}_M^0(X)$)  such that the conditions in \cref{thm:vanishing2} are fulfilled. Moreover, by the structure of the linear Shafarevich morphism, for  \emph{almost every}  smooth fiber  of $f$, the conditions in \cref{prop:abelian} are satisfied.  
 
 Therefore, if there exists a non-trivial $\alpha\in H^{p,0}_{(2)}(\widetilde{X})$ for some $p\in \{0,\ldots,\dim X-1\}$,  we can  apply \cref{item:partial}  to show that $p\geq m:=\dim Y$, and use \cref{item:d-closed} to obtain a non-trivial $L^2$ holomorphic $(p-m)$-form on the universal cover of  at least one  smooth fiber  of $f$. Finally, we apply \cref{prop:abelian} to obtain a contradiction.  
	
 	\Cref{subsec:semisimple} is covered by \Cref{subsec:linear}, and readers can skip it if they prefer to proceed directly to the proof of the general case of \cref{main:kollar}. 
 	\subsection{Two technical results}
 	 	Note that in \eqref{eq:defbeta}, we construct certain  1-forms $\beta_{\tau_i}$ on $\widetilde{X}$ associated with the non-archimedean representations $\tau_i$   in \cref{thm:Sha}. For the $\bC$-VHS $\cL$, there is also another way to construct  similar 1-forms. This was established by Eyssidieux in \cite{Esy97} and we recollect some facts  therein. 
 	\begin{proposition}[\protecting{\cite[Proposition 4.5.1]{Esy97}}]\label{prop:eys}
 		Let $X$ be a smooth projective variety and let $\cL$ be a $\bC$-VHS on $X$. Let $\cD$ be the period domain of $\cL$ and let $f:\widetilde{X}\to \cD$ be the period map. Let $q:\cD\to \cR$ be the natural quotient where $\cR$ is the corresponding Riemannian symmetric space of $\cD$. Define $\omega_\cL:=\sqrt{-1}{\rm tr}(\theta\wedge\theta^*)$ as in \cref{thm:Sha}, which is a smooth closed positive $(1,1)$-form. Then there exists a smooth  function $\psi_{\cD}:\cR\to \bR_{>0}$ and  constants $c>0$   depending  only on $\cD$  such that  the function $\phi:=\psi_{\cD}\circ q\circ f$ is smooth and plurisubharmonic and we have
 		\begin{align}\label{eq:eys}
 			\sqrt{-1}\partial\phi\wedge\db\phi&\leq \pi_X^*\omega_\cL,\\\label{eq:eys2}
 			c\pi_X^*\omega_\cL&\leq 	\hess \phi.
 		\end{align}  \qed
 	\end{proposition} 
 	Therefore, we have the following consequence.
 	\begin{corollary}\label{cor:eta}
	Let $X$ be a smooth projective variety and let $\cL$ be a $\bC$-VHS on $X$. Fix a Kähler metric $\omega$ on $X$. For  $\bC$-VHS $\cL$ on $X$, there exists a smooth 1-form $\eta_\cL$ on the universal cover $\widetilde{X}$ and a constant $C > 0$ such that, $d\eta_\cL$ is a real positive $(1,1)$-form satisfying
 \begin{align}
	|\eta_\cL|_{\pi_X^*\omega} \leq C, \quad d\eta_\cL \geq \pi_X^*\omega_\cL.
\end{align}  
\end{corollary}
\begin{proof}
	Define
	\[
	\eta_\cL := \frac{\sqrt{-1}  \db \phi}{c},
	\]
	where $\phi$ and the constant $c > 0$ are as in \cref{prop:eys}. Since there exists a constant $c_1 > 0$ such that $\omega \geq c_1 \omega_\cL$, it follows from \eqref{eq:eys} that
	\[
	|\eta_\cL|_{\pi_X^*\omega} \leq \frac{c_1}{c}.
	\]
	Moreover, by \eqref{eq:eys2}, we have
	\[
	d\eta_\cL \geq \pi_X^*\omega_\cL.
	\]
\end{proof}
	\subsection{Case of semi-simple local systems}\label{subsec:semisimple} 
		We will use   techniques in proving the reductive Shafarevich conjecture in \cite{DY23,Eys04}.  We summarize the main results needed in  proving \cref{thm:main3} as follows. 
	\begin{theorem}[\protecting{\cite[Proof of Theorem 4.31]{DY23}}]\label{thm:Sha}
		Let \( X \) be a smooth projective variety, and let \( \varrho: \pi_1(X) \to \GL_N(\bC) \) be a semisimple representation. If \( \varrho \) is generically large, then there exist:
		\begin{enumerate}
			\item a family of Zariski dense representations \( \{\tau_i: \pi_1(X) \to G_i(K_i)\}_{i=1,\ldots,\ell} \), where each \( G_i \) is a reductive group over a non-archimedean local field \( K_i \) of characteristic zero;
			\item a \( \bC \)-VHS \( \cL \) on \( X \),
		\end{enumerate}
		such that
		\[
		T_{\tau_1} + \cdots + T_{\tau_\ell} + \sqrt{-1} \, \mathrm{tr}(\theta \wedge \theta^*)
		\]
		is a closed positive \( (1,1) \)-current on \( X \), which is smooth and strictly positive over a non-empty analytic open subset \( X^\circ \subset X \).
		
		Here:
		\begin{itemize}
			\item \( T_{\tau_i} \) denotes the canonical current on \( X \) associated with \( \tau_i \), as defined in \cref{def:canonical};
			\item \( \theta \) is the Higgs field of the Hodge bundle associated with \( \cL \), and \( \theta^* \) is its adjoint with respect to the Hodge metric. \qed
		\end{itemize}
	\end{theorem}
	\begin{theorem}\label{thm:main3}
		Let $X$ be a smooth projective $n$-fold. Let $\varrho:\pi_1(X)\to \GL_N(\bC)$ be a semisimple representation. 
  If $\varrho$ is  generically large,  then  
 $H^{p,0}_{(2)}(\widetilde{X})=0$   for  $0\leq p\leq n-1$.
	\end{theorem} 
	\begin{proof} 
By \cref{thm:Sha}, there exist	Zariski dense   representations $\{\tau_i:\pi_1(X)\to G_i(K_i)\}_{i=1,\ldots,\ell}$  where each $G_i$  is a reductive algebraic group over a non-archimedean field $K_i$, along with a $\bC$-VHS $\cL$  on $X$ satisfying the stated properties.   
We denote by $\omega_\cL:=\sqrt{-1}{\rm tr}(\theta\wedge\theta^*)$, that  is a smooth $(1,1)$-form defined in \cref{thm:Sha}.  We fix a K\"ahler form $\omega_X$ on $X$, and, by abuse of notation, also denote by $\omega_X$ its pullback to the universal cover $\widetilde{X}$.
		
	 In what follows, for any form $\eta$ we shall denote by $|\eta|$ its norm with respect to $\omega_X$.   		Let $u_i:\widetilde{X}\to \Delta(G_i)$ be the $\tau_i$-equivariant pluriharmonic map whose existence is ensured by \cite{GS92}, where $\Delta(G_i)$ is the Bruhat-Tits building of $G_i$.  By \eqref{eq:defbeta},  if we define
		\begin{align}\label{eq:beta}
		\beta_i(x)=\sqrt{-1} \db d^2_{\Delta(G_i)}(u_i(x),u_i(x_0)),
		\end{align}   then  by \cref{lem:distance}, $\beta_i$ has $L_{\rm loc}^1$-coefficients and  is  smooth outside a set of zero Lebesgue measure.  Moreover,     there exists a constant $c_1>0$  with 
		$
	|\beta_i(x)|\leq_{\rm a.e.}c_1 (d_{\widetilde{X}}(x,x_0)+1) 
		$  
		for any $i$.

		Let $\eta_\cL$ be  smooth 1-form in $\widetilde{X}$ defined in \cref{cor:eta}. Then we have 
		\begin{align} \label{eq:between}
|\eta_\cL|\leq c_2,\quad d\eta_\cL\geq \pi_X^*\omega_\cL 
		\end{align}  
		for some positive constant $ c_2$.    
	Define 
 $\beta:=\eta_\cL+\sum_{i=1}^{\ell}\beta_i$.  Then $\beta$ has $L_{\rm loc}^1$-coefficients. By \cref{lem:distance} and \eqref{eq:between}, it satisfies 
\begin{align}\label{eq:dsub}
	  |\beta(x)|\leq_{\rm a.e.} c_3(1+d_{\widetilde{X}}(x, x_0)) 
\end{align} 
	 for some constant $c_3>0$.   	By \cref{lem:equal} and \eqref{eq:between}, 	  we have
	 \begin{align*} d\beta\geq   \pi_X^*(\sum_{i=1}^{\ell}T_{\tau_i}+\omega_{\cL} ). 
	 \end{align*}   
 By  \cref{thm:Sha}, $\sum_{i=1}^{\ell}T_{\tau_i}+\omega_{\cL}$ is closed positive $(1,1)$-current, that is smooth and strictly positive over an analytic   open subset of $X$.  Therefore,   the 1-form $\beta$ on $\widetilde{X}$ satisfies the assumptions of \cref{thm:vanishing2}, with $f$ taken to be the identity map. We thus obtain the desired $L^2$ vanishing result.
	\end{proof} 
	\begin{remark}\label{rem:final}
		If we compare the proof of \cref{thm:main3}  with that  of the reductive Shafarevich conjecture in \cite{DY23,Eys04}, we observe striking similarities in their approaches. 
		Indeed, in \cite[Proposition 4.1.1]{Eys04}, Eyssidieux proved the following result: let $X$ be   a compact K\"ahler normal variety. If there exist a continuous plurisubharmonic function $\phi:\widetilde{X}\to \bR_{>0}$ and a positive closed $(1,1)$-current $T$ on $X$ with continuous potential such that  $\{T\}$ is a K\"ahler class and  $\hess\phi\geq \pi_X^*T$, then $\widetilde{X}$ is Stein. Therefore, if $X$ is a smooth projective variety endowed with a  semisimple and large representation $\varrho:\pi_1(X)\to \GL_N(\bC)$, then we can prove that $\widetilde{X}$  is Stein using \cref{thm:Sha} as follows. 
		
		Consider the continuous function
		$$
		\phi_0:=\sum_{i=1}^{\ell}	 d^2_{\Delta(G_i)}(u_i(x),u_i(x_0))+\frac{\phi}{c_2}
		$$
		on $\widetilde{X}$, where $u_i$ and $\phi$ are defined  in the proof of \cref{thm:Sha}. We then have
		$$
		\hess\phi_0\geq \pi_X^*(\sum_{i=1}^{\ell}T_\ell+ \omega_\cL).
		$$
		Note that $\{ \sum_{i=1}^{\ell} T_\ell+\omega_\cL\}$ is a K\"ahler class in $X$ if $\varrho$ is large by \cref{thm:Sha}. Therefore, by the above Eyssidieux's criterion, we conclude that $\widetilde{X}$ is Stein. 	
	\end{remark}

\subsection{Proof of \cref{main:kollar}}\label{subsec:linear}
In this subsection we will prove \cref{main:kollar}. 
\begin{theorem} \label{thm:main2}
	Let $X$ be a smooth projective $n$-fold. Let $\varrho:\pi_1(X)\to \GL_N(\bC)$ be a generically large  representation.  Then
  $H^{p,0}_{(2)}(\widetilde{X})=0$   for  $0\leq p\leq n-1$.
\end{theorem}
\begin{proof}
	We fix a smooth K\"ahler metric $\omega_X$ on $X$. To lighten the notation, we write  $\omega_X$   instead of   $\pi_X^*\omega_X$ abusively.

	\noindent {\bf Step 1.} Consider the Betti moduli space \( M := M_{\mathrm{B}}(X, \GL_N)\). Let 
$
	\mathrm{sh}_M : X \to \mathrm{Sh}_M(X)
$ 
	denote the reductive Shafarevich morphism associated with \( M \) (cf.\ \cite{Eys04, DY23}). In other words, it is a morphism over a normal projective variety such that, for any closed subvariety \( Z \subset X \),  \( \mathrm{sh}_M(Z) \) is a point if and only if, for every \emph{semisimple} representation \( \sigma: \pi_1(X) \to \mathrm{GL}_N(\mathbb{C}) \), the image \( \sigma(\mathrm{Im}[\pi_1(Z) \to \pi_1(X)]) \) is finite.
	
	By \cite[Proof of Theorem 3.29]{DY23}, after replacing \( X \) by a finite étale cover, there exist:
	\begin{itemize}
		\item a family of Zariski dense representations \( \{\tau_i : \pi_1(X) \to G_i(K_i)\}_{i=1,\ldots, \ell} \), where each \( G_i \) is a reductive group over a non-archimedean local field \( K_i \) of characteristic zero;
		\item a complex variation of Hodge structure \( \cL \);
		\item a smooth semi-positive \( (1,1) \)-form \( \omega_M \) on \( \mathrm{Sh}_M(X) \);
		\item for each \( i \), a closed positive \( (1,1) \)-current \( S_i \) on \( \mathrm{Sh}_M(X) \).
	\end{itemize}
	These data satisfy
	\[
	T_{\tau_i} = \mathrm{sh}_M^* S_i, \quad \text{and} \quad \sqrt{-1} \, \mathrm{tr}(\theta \wedge \theta^*) = \mathrm{sh}_M^* \omega_M,
	\]
	where \( T_{\tau_i} \) denotes the canonical current associated with \( \tau_i \) in \cref{def:canonical}. 

Moreover, there exists a non-empty analytic open subset \( U \subset \mathrm{Sh}_M(X) \) such that the positive \( (1,1) \)-current
$
\sum_{i=1}^\ell S_i + \omega_M
$ 
is smooth and strictly positive on \( U \). We denote by \( f : X \to Y \) the reductive Shafarevich morphism \( \mathrm{sh}_M \). 
	
	\medspace
	
		\noindent {\bf Step 2.}  
	For each $i$, let  $\beta_i$ be  the 1-form on $\widetilde{X}$ defined in \eqref{eq:beta} associated with the representation $\tau_i$. By \cref{lem:equal}, it is smooth outside a set of Lebesgue measure zero and satisfies 
	\begin{align}\label{eq:f1}
		d\beta_i\geq \pi_X^*T_{\tau_i}.
	\end{align}	 	By \cref{lem:distance}, 	 there exists  	a number $c_0>0$    such that for any $i\in\{1,\ldots,\ell\}$, we have
	\begin{align}\label{eq:sublinear4}
		|\beta_i(x)|_{\omega_X} \leq_{\rm a.e.} c_0(1+d_{\widetilde{X}}(x, x_0)).
	\end{align}   
	
	Let $\eta_\cL$ be  smooth 1-form in $\widetilde{X}$ defined in \cref{cor:eta}. Then we have 
	\begin{align} \label{eq:estimate5}
		|\eta_\cL|\leq c_1,\quad d\eta_\cL\geq \pi_X^*\sqrt{-1}\mathrm{tr}(\theta \wedge \theta^*) 
	\end{align}  
	for some positive constant $ c_1$.    
	Write 
	$$\beta:=\beta_1+\cdots+\beta_\ell+\eta_\cL. $$ Then $\beta$ has $L_{\rm loc}^1$-coefficients and smooth  outside a set of Lebesgue measure zero.  By \eqref{eq:sublinear4} and \eqref{eq:estimate5}, we have
	\begin{align}\label{eq:bound}
		|\beta(x)|_{\omega_X}\leq_{\rm a.e.} c_2(1+d_{\widetilde{X}}(x, x_0)),
	\end{align} 
	for some constant $c_2>0$. 
	By  \eqref{eq:f1} and \eqref{eq:estimate5}, one has \begin{align}\label{eq:kahlerpara}
		\pi_X^*f^*( \sum_{i=1}^\ell S_i + \omega_M)=	T_{\tau_1}+\cdots+T_{\tau_\ell}+\sqrt{-1}{\rm tr}(\theta\wedge\theta^*)\leq  	  d\beta.   
	\end{align}  
	Recall that there exists a non-empty analytic open subset \( U \subset Y \) such that the positive \( (1,1) \)-current
	$
	\sum_{i=1}^\ell S_i + \omega_M
	$
	is smooth and strictly positive on \( U \).
	Therefore,   the 1-form $\beta$ satisfies   the conditions in \cref{thm:vanishing2}.  
	
	\medspace
	
	\noindent {\bf Step 3.} Let \( H \) be the intersection of the kernels of all representations \( \pi_1(X) \to \GL_N(\bC) \). By \cite[Theorem~1]{EKPR12}, the Shafarevich morphism 
	\[
	\mathrm{sh}_H: X \to \mathrm{Sh}_H(X)
	\]
	associated with \( H \) exists. That is, \( \mathrm{sh}_H \) is a morphism to a normal projective variety with connected fibers, characterized by the following property: for any subvariety \( Z \subset X \), the image 
	\[
	\mathrm{Im}[\pi_1(Z) \to \pi_1(X)/H]
	\]
	is finite if and only if \( \mathrm{sh}_H(Z) \) is a point.
	
	In particular, since the representation \( \varrho: \pi_1(X) \to \GL_N(\bC) \) is generically large, it follows that \( \mathrm{sh}_H \) is a birational morphism. Moreover, the morphism \( \mathrm{sh}_M \) factors through \( \mathrm{sh}_H \)
	\begin{equation*}
		\begin{tikzcd}
		X\arrow[r,"{\rm sh}_H"] \arrow[dr, "{\rm sh}_M"' ]& \mathrm{Sh}_H(X)\arrow[d,"q"]\\
&		 \mathrm{Sh}_M(X)
		\end{tikzcd}
	\end{equation*}
	
	By \cite[\S5]{EKPR12}, for a general smooth fiber \( Z \) of \(  \mathrm{sh}_H \), there exists an abelian variety \( A \) and a morphism 
	\[
	a: Z \to A
	\]
	such that the Stein factorization of \( a \) coincides with the restriction \( \mathrm{sh}_H|_Z: Z \to \mathrm{sh}_H(Z) \). Since \( \mathrm{sh}_H \) is birational, we conclude that
	\[
	\dim Z = \dim a(Z).
	\]
	
	\medspace

	\medspace
	
	\noindent {\bf Step 4.} We will apply \cref{thm:vanishing2} and use its notations as defined therein, without re-explaining their meanings.  Assume by contradiction that, for some $p\in \{0,\ldots,n-1\}$, there is a non-trivial $\alpha\in H^{p,0}_{(2)}(\widetilde{X})$. By \cref{item:partial},  we have $n>m$ and $p\geq m$. Furthermore, over the analytic open subset $Y^\circ$ of  $Y^{\rm reg}$, 
	we have   \begin{align*} 
		\alpha|_{\widetilde{X}^\circ}\in H^0(\widetilde{X}^\circ, \tilde{f}^*\Omega_{\widetilde{Y}^\circ}^{m}\otimes \Omega_{\widetilde{X}^\circ}^{p-m}), 
	\end{align*}
	where $\widetilde{Y}^\circ:= \pi_Y^{-1}(Y^\circ)$ and $\widetilde{X}^\circ:=\tilde{f}^{-1}(\widetilde{Y}^\circ)$.   We pick any  $y_0\in  {Y}^\circ$, and choose a simply connected coordinate open subset $(V;w_1\ldots,w_m)$ centered at $y_0$.  We  abusively denote by  $V$ a connected component of $\pi_Y^{-1}(V)$. Then $dw_1\wedge \cdots\wedge dw_m$ is a nonwhere-vanishing holomorphic $m$-form on $V$. Let $\Omega$ be a connected component of $\tilde{f}^{-1}(V)$.  We denote by $g:\Omega\to V$ the restriction of $\tilde{f}$  on $\Omega$. Then $g$ is a submersion with connected fibers. For any $y\in V$, by Step 3 in the proof of \cref{thm:vanishing2},  $dw_1\wedge \cdots\wedge dw_m$  induces  a unique holomorphic $(p-m)$-form $\alpha_{y}$ on  $g^{-1}(y)$ defined in \eqref{eq:express}  such that  $$\alpha_y\wedge \tilde{f}^*(dw_1\wedge \cdots\wedge dw_m)=\alpha.$$
	Then by \eqref{eq:express2} and \eqref{eq:express}   together with the Fubini theorem,  we have
	\begin{align*}
		0\leq &\int_{V} \left(\int_{g^{-1}(y)}\sqrt{-1}^{(p-m)^2}\alpha_{y}\wedge \overline{\alpha_{y}}\wedge (\omega_X|_{g^{-1}(y)})^{n-p}\right) idw_1\wedge d\bar{w}_1\wedge\cdots\wedge  idw_m\wedge d\bar{w}_m\\
		&	 = \int_\Omega \sqrt{-1}^{p^2}\alpha\wedge \bar{\alpha}\wedge \omega_X^{n-p} \leq \int_{\widetilde{X}}|\alpha|^2\dvol<+\infty. 
	\end{align*} 
	Therefore,  there is a subset $Z$ of $V$ with zero Lebesgue measure,  such that for any $y\in V\backslash Z$, we have
	\begin{align}\label{eq:L2}
		\int_{g^{-1}(y)}\sqrt{-1}^{(p-m)^2}\alpha_{y}\wedge \overline{\alpha_{y}}\wedge (\omega_X|_{g^{-1}(y)})^{n-p}<\infty.
	\end{align} 
	Thus, we   construct an $L^2$ holomorphic $(p-m)$-form $\alpha_y$ on $g^{-1}(y)$ for any $y\in V\backslash Z$, which is also $d$-closed by \cref{item:d-closed}. We note that  for any $x\in \tilde{f}^{-1}(y)$,   $\alpha_y(x)=0$ if and only if $\alpha(x)=0$.
	
	Denote \( X_y := f^{-1}(y) \). By Step 3, if we choose a general point \( y \in V \setminus Z \), then there exists a morphism \( a: X_y \to A \) to an abelian variety \( A \) such that \( \dim X_y = \dim a(X_y) \). Let \( \widetilde{X_y}' \) be a connected component of the fiber product \( X_y \times_{A} \widetilde{A} \),  where $\widetilde{A}$ is the universal cover of $A$, and let \( \pi_y': \widetilde{X_y}' \to X_y \) denote the covering map. By \cite[Lemma 5.1]{EKPR12}, the map \( g^{-1}(y) \to X_y \) is an infinite Galois covering that factors through \( \pi_y': \widetilde{X_y}' \to X_y \).
	
	The conditions of \cref{prop:abelian} are therefore satisfied, and we deduce that \( \alpha_y \equiv 0 \). Since \( y \) is a general point in \( V \setminus Z \), it follows that \( \alpha(x) = 0 \) for almost all \( x \in \Omega \). By continuity, we conclude that \( \alpha(x) = 0 \) everywhere. This yields a contradiction. Therefore, we conclude that
	\[
	H^{p,0}_{(2)}(\widetilde{X}) = 0 \quad \text{for all } p \in \{0, \ldots, n-1\}.
	\]
	The theorem is proved.
\end{proof} 
We will apply \cref{thm:main2} to prove \cref{main:kollar}.
\begin{proof} [Proof of \cref{main:kollar}]
	We fix a smooth K\"ahler metric $\omega$ on $X$. 	Let 	$\cH_{(2)}^{n,q}(\widetilde{X})	$ be the $L^2$-harmonic $(n,q)$-forms with respect to the metric $\pi_X^*\omega$. By the Lefschetz theorem in  \cite[Theorem 1.2.A]{Gr}),  for any $q\in \{1,\ldots,n\}$, 
	\begin{align*}
		\cH_{(2)}^{n-q,0}(\widetilde{X})&\to  	\cH_{(2)}^{n,q}(\widetilde{X})\\\alpha&\mapsto \omega^q\wedge\alpha
	\end{align*}
	is an isomorphism. Since we have an isomorphism $\cH_{(2)}^{n,q}(\widetilde{X})\simeq H_{(2)}^{n,q}(\widetilde{X})$,  this establishes that $ H_{(2)}^{n,q}(\widetilde{X})=0$ for $q\in \{1,\ldots,n\}$.

	We denote by $\Gamma=\pi_1(X)$ and $\dim_{\Gamma}H^{n,q}_{(2)}(\widetilde{X})$ the Von Neumann dimension of $H^{n,q}_{(2)}(\widetilde{X})$ (cf. \cite{Ati76} for the definition). 
	By Atiyah's $L^2$-index theorem along with \cref{item:vanishing}, we have
	\begin{align} \label{eq:euler}	\chi(X,K_X)=\sum_{q=0}^{n}(-1)^q\dim_{\Gamma}H^{n,q}_{(2)}(\widetilde{X})= \dim_{\Gamma}H^{n,0}_{(2)}(\widetilde{X})\geq 0.
	\end{align} 
	\Cref{item:Euler} is proved.

	If the strict  inequality   \eqref{eq:euler}  holds, then $H^{n,0}_{(2)}(\widetilde{X})\neq 0$. We then apply \cite[Corollary 13.10]{Kol95} to conclude that $K_X$ is big. \Cref{item:strict} follows.  The theorem is proved. 
\end{proof}


\end{document}